\documentclass[10pt,reqno]{amsart}
\usepackage{amsmath,amscd,amsfonts,amsthm,amssymb,latexsym,mathrsfs,url}
\usepackage{tikz}

\usepackage{hyperref}
\hypersetup{
    bookmarks=true,         
    unicode=false,          
    pdftoolbar=true,        
    pdfmenubar=true,        
    pdffitwindow=false,     
    pdfstartview={FitH},    
    pdftitle={My title},    
    pdfauthor={Author},     
    pdfsubject={Subject},   
    pdfcreator={Creator},   
    pdfproducer={Producer}, 
    pdfkeywords={keyword1} {key2} {key3}, 
    pdfnewwindow=true,      
    colorlinks=true,       
    linkcolor=blue,          
    citecolor=blue,        
    filecolor=blue,      
    urlcolor=blue      
    }

\usepackage{skak}
 \numberwithin{equation}{section}
\theoremstyle{plain}
   \newtheorem{theorem}{Theorem}[section]
   \newtheorem{proposition}[theorem]{Proposition}
   
   \newtheorem{corollary}[theorem]{Corollary}

   \newtheorem{question}{Question}
\theoremstyle{definition}

\theoremstyle{remark}
   \newtheorem{remark}[theorem]{Remark}

\newcommand{\NN}{\mathbb{N}}
\newcommand{\XX}{\mathcal{X}}
\newcommand{\YY}{\mathcal{Y}}

\newcommand{\xx}{\mathbf{x}}
\newcommand{\yyy}{\mathbf{y}}

\newcommand{\ZZ}{\mathbb{Z}}
\newcommand{\OO}{\mathcal{O}}

\newcommand{\RR}{\mathbb{R}}

\newcommand{\sym}{\mathfrak{S}}

\renewcommand{\Im}{{\rm Im}}

\def\multiset#1#2{\ensuremath{\left(\kern-.3em\left(\genfrac{}{}{0pt}{}{#1}{#2}\right)\kern-.3em\right)}}

\def\newop#1{\expandafter\def\csname #1\endcsname{\mathop{\rm
#1}\nolimits}}

\newop{inv}
\newop{diag}
\newop{AB}
\newop{DB}
\newop{B}
\newop{AT}
\newop{DT}
\newop{peak}
\newop{rank}
\newop{des}
\newop{Sym}
\newop{exc}
\newop{sign}
\newop{wcr}
\newop{cro}
\newop{pc}
\newop{zc}
\newop{nc}
\newop{yy}

\title[ASEP and Eulerian polynomials]{Multivariate Eulerian polynomials \\ and exclusion processes} 
\author[P. Br\"and\'en, M. Leander and M. Visontai]{Petter Br\"and\'en, Madeleine Leander and Mirk\'o Visontai}
\thanks{The first author is a Royal Swedish Academy of Sciences Research Fellow
  supported by a grant from the Knut and Alice Wallenberg
  Foundation. His research is also supported by the G\"oran Gustafsson Foundation. The third author was supported by the Knut and Alice Wallenberg
  Foundation.}

\begin{document}
\begin{abstract}
We give a new combinatorial interpretation of the stationary distribution of the (partially) asymmetric exclusion process on a finite number of sites  in terms of decorated alternative trees and colored permutations. The corresponding expressions of the multivariate partition functions are then related  to multivariate generalizations of Eulerian polynomials for colored permutations considered recently by N. Williams and the third author, and others. We also discuss stability-- and negative dependence properties satisfied by the partition functions. 
\end{abstract}
\maketitle

\thispagestyle{empty}

\section{Introduction}
The \emph{Eulerian polynomial}, $A_n(x)$, $n \in \NN$, may be defined as the generating polynomial for the \emph{descent statistic} over the symmetric group $\sym_n$;
$$
A_n(x) = \sum_{\sigma \in \sym_n} x^{\des(\sigma)+1},
$$
where $\des(\sigma)= |\{ i : \sigma(i)>\sigma(i+1)\}|$. Another important statistic which has the same distribution as descents is the number of \emph{excedances}, 
 $\exc(\sigma)= |\{ i : \sigma(i) >i\}|$. 
Eulerian polynomials are among the most studied families of polynomials in combinatorics. There are several multivariate extensions of the Eulerian polynomials. We are interested in the one that refines 
the excedance statistic over permutations by the position of excedances. This multivariate refinement was used in 
conjunction with stable polynomials in recent papers to solve the monotone column permanent conjecture \cite{Mon}. Similar methods were employed 
to generalize properties---such as recurrences, and zero location---for several variants of Eulerian polynomials 
 for Stirling permutations \cite{HV2}, barred multiset permutations \cite{Chen13}, signed and colored permutation \cite{VW13}.

In statistical mechanics the same multivariate Eulerian polynomials appear in connection with an important and much studied  Markov process called the 
\emph{asymmetric exclusion process} (ASEP), which models particles hopping left and right  on a one-dimensional lattice, see \cite{CW1,CW2,ansatz,Lig} and the references therein. 
The ASEP has  a unique stationary distribution, and Corteel and Williams \cite{CW1} showed that its partition function is a $q$--analog of the above mentioned multivariate 
Eulerian polynomial (whenever $\alpha=\beta=1$ and $\gamma=\delta=0$, where $\alpha, \beta, \gamma$ and $\delta$ are as described in Section \ref{secEP}). \\[1ex]
 
In this paper we show that the multivariate partition function of the stationary distribution of the ASEP with parameters $\alpha = r/(r-1)$ and $\beta = r$ 
is a $q$-analog of the multivariate Eulerian polynomials for $r$-colored permutations recently introduced in \cite{VW13}. This contains the signed permutations as special case for $r=2$. We also give a new  combinatorial interpretation of the stationary distribution in terms of colored permutations for all $q, \alpha, \beta \geq 0$ and $\gamma=\delta=0$. Previous combinatorial interpretations of the stationary distributions have been proved using either permutation tableaux \cite{CW1}, or staircase tableaux \cite{CW2}, while our proof uses directly the (matrix) ansatz for the ASEP of Liggett \cite{Lig}, which was later rediscovered by Derrida,  Evans, Hakim and Pasquier \cite{ansatz}.

Our methods extend certain properties of the excedance set statistic studied for the case of permutations by Ehrenborg and Steingr\'imsson \cite{ES00}, as well as the 
plane alternative trees of Nadeau \cite{Nad11}, to the case of $r$-colored permutations. 

We also point to negative dependence properties and zero restrictions of the multivariate partition function satisfied by stationary distribution for $q=1$ which follow from general theorems obtained by Borcea, Liggett and the first author in \cite{BBL}, and Wagner \cite{Wa3}. These negative dependence properties and zero restrictions generalize recent results of Hitczenko and Janson \cite{HiJa}. We speculate in what negative dependence properties may hold for $q \neq 0$.

\section{Exclusion processes}
\label{secEP}
We focus on  a class of exclusion processes that model particles  jumping on a finite set of sites, labeled by $[n] := \{1,2,\ldots, n\}$. Given a matrix $Q= (q_{ij})_{i,j=1}^n$  of nonnegative numbers and vectors $b=(b_i)_{i=1}^n$ and $d=(d_i)_{i=1}^n$ of nonnegative numbers, define a continuous time Markov chain on $\{0,1\}^n$ as follows. Let $\eta \in \{0,1\}^n$ represent the configuration of the particles, with $\eta(i)=1$ meaning that site $i$ is occupied, and $\eta(i)=0$ that site $i$ is vacant. Particles at occupied sites jump to vacant sites at specified rates. More precisely, these are the transitions in the Markov chain:
\begin{itemize}
\item[(J)] A particle jumps from site $i$ to site $j$ at rate $q_{ij}$: The configuration $\eta$ is unchanged unless  $\eta(i)=1$ and $\eta(j)=0$, and then only $\eta(i)$ and $\eta(j)$ are exchanged. 
\item[(B)] A particle at site $i$ is created (is born) at rate $b_i$: The configuration $\eta$ is unchanged unless  $\eta(i)=0$, and then only $\eta(i)$ is changed from a zero to a one. 
\item[(D)] A particle at site $i$ is annihilated (dies) at rate $d_i$: The configuration $\eta$ is unchanged unless  $\eta(i)=1$, and then only $\eta(i)$ is changed from a one to a zero. 
\end{itemize}

The (multivariate) \emph{partition function} of a discrete probability measure $\mu$ on $\{0,1\}^n$ is the polynomial in $\RR[x_1,\ldots, x_n]$ defined by
\begin{equation}\label{partitioneq}
Z_\mu(\xx)= \sum_{\eta \in \{0,1\}^n} \mu(\eta) \xx^\eta:=\sum_{\eta \in \{0,1\}^n} \mu(\eta) x_1^{\eta(1)}\cdots x_n^{\eta(n)}.
\end{equation}
Hence a discrete probability measure can be recovered from its partition function. 

A special case of the Markov chain described by (J), (B) and (D) above which has been much studied by combinatorialists  is the ASEP (on a line). Here 
$$
q_{ij} = 
\begin{cases}
1 &\mbox{ if } j=i+1, \\
q &\mbox{ if } j=i-1, \mbox{ and }\\
0 &\mbox{ if } |j-i|>1, 
\end{cases}
$$
where $q\geq 0$ is a parameter. 
Moreover, particles are only allowed to leave and enter at the ends $i=1, n$ ($b_1= \alpha, d_1=\gamma, b_n=\delta, d_n=\beta$). The stationary distributions of the ASEP have been explicitly solved by \cite[Theorem~3.2]{Lig} and later by the Matrix Ansatz of \cite{ansatz}, and also by beautiful combinatorial models such as permutation tableaux, staircase tableaux, and alternative tableaux \cite{CW1,CW2,Vi}. 
The following theorem is essentially the Matrix Ansatz of \cite{ansatz} for $\gamma =\delta =0$. 
\begin{theorem}\label{az}
Let $\alpha, \beta, q \geq 0$ and $\xi > 0$, and let $\{0,1\}^*$ be the set of words of finite length of zeros and ones. 
Define a function $\langle \cdot \rangle : \{0,1\}^* \mapsto \RR$ recursively by $\langle \varepsilon\rangle =1$ if $\varepsilon$ is the empty word, and 
\[
\langle u10v\rangle =q\langle u01v\rangle + \alpha\beta \xi \langle u1v\rangle + \alpha\beta \xi \langle u0v\rangle , \quad \langle 0v\rangle=\beta\xi\langle v\rangle, \quad \langle u1\rangle =\alpha\xi\langle u\rangle , 
\]
for any $u,v \in \{0,1\}^*$. 

Then the partition function of the ASEP on $n$ sites with parameters $\alpha, \beta, q \geq 0$ and $\gamma=\delta = 0$ is equal to a constant multiple of 
$$
\sum_{\eta \in \{0,1\}^n} \langle \eta \rangle \xx^\eta, 
$$
where we identify $\eta$ with the corresponding word $\eta(1)\cdots \eta(n) \in \{0,1\}^*$. 
\end{theorem}

In the next  few sections we use Theorem \ref{az} to give a purely combinatorial interpretation of the partition function of the
stationary distribution of the ASEP in terms of permutation statistics.

\section{Alternative trees}
Alternative trees were introduced by Nadeau \cite{Nad11} in connection with alternative tableaux, which were used  by Viennot \cite{Vi} as combinatorial model for the ASEP. Throughout this section and the next, $S$ is a finite, nonempty and totally ordered set. An unordered rooted tree with vertex set $S$ is called an  \emph{alternative tree}  if the following three conditions are satisfied
\begin{itemize}
\item[(i)] The root is either $\min S$ or $\max S$. 
\item[(ii)] If a vertex is larger than its parent, then it is larger than all of its descendants.
\item[(iii)] If a vertex is smaller than its parent, then it is smaller than all of its descendants.
\end{itemize}
See Fig.~\ref{AltTree} for an example. 
\begin{figure}[htp]
$$
\begin{tikzpicture}[scale=0.75, inner sep=1pt]
  \node (0) at (2,3) {$0$};
  \node (1) at (2,1) {$1$};
    \node (2) at (3,0) {$2$};
    \node (3) at (0,1) {$3$};
    \node (4) at (3,1) {$4$};
    \node (5) at (4,1) {$5$};
  \node (6) at (1,0) {$6$};
  \node (7) at (3,2) {$7$};
   \node (8) at (1,2) {$8$};
  \draw (3) -- (8) -- (0) -- (7) -- (1) -- (6);
   \draw (1) -- (2);
   \draw (7) -- (4);
   \draw (7) -- (5);
\end{tikzpicture}
$$ 
\caption{ An alternative tree on the set $[0,8]$ corresponding to the marked cycle $((386214570),0)$, or the permutation $\sigma=
(38)(621457)$. } 
\label{AltTree}
\end{figure}
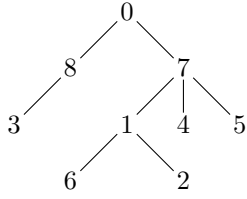

If we write a cycle as a word $a_1 \cdots a_k$, we mean $a_1 \to a_2 \to \cdots \to a_k \to a_1$. 

A \emph{marked cycle} on $S$ is a pair $(\sigma, s)$ where $\sigma$ is a cycle on $S$ and 
$s \in S$ is either the maximum or minimum of $S$.  We write a marked cycle $(\sigma,s)$ as a word
\begin{equation}\label{ws}
\sigma = Ws=W_kx_k\cdots W_2x_2W_1x_1s, 
\end{equation}
where the $x_i$'s are defined as follows.  
\begin{itemize}
\item 
If $s=\min{S}$, then $x_k>\cdots >x_1$ are the \emph{right-to-left maxima} of the word $W$, that is, $x_i$ is larger than all letters in $W_i W_{i-1} x_{i-1}\cdots W_1x_1$ for all $1\leq i \leq k$. 
\item If $s=\max{S}$, then $x_k<\cdots <x_1$ are the \emph{right-to-left minima} of $W$, that is, $x_i$  is smaller than all letters in $W_i W_{i-1} x_{i-1}\cdots W_1x_1$ for all $1\leq i \leq k$.
\end{itemize}
Note that the words $W_i$ are implicitly defined. 

We will  recursively describe a map $T$ from marked cycles on $S$ to alternative trees on $S$, and then extend the construction to a bijection between permutations and alternative trees. 
 If $\sigma=s$ is a marked cycle on one letter, then $T(\sigma, s)$ is a vertex labeled with $s$. Otherwise write $\sigma =Ws$ as in \eqref{ws}.  The root of $T(\sigma,s)$ is $s$, the children of $s$ are $x_1,\ldots, x_k$, and the subtree with $x_i$ as a root is $T(W_i x_i, x_i)$, where we consider $(W_i x_i, x_i)$ a marked cycle on the set of letters of the word $W_ix_i$ for all $1\leq i \leq k$, see Fig.~\ref{AltTree}.

By  a straightforward induction argument, it follows that the  map $T$ described above is a bijection between alternative trees on $S$ and marked cycles on $S$.

\begin{proposition}\label{cycleprop}
Let $(\sigma , s)$ be a marked cycle of length greater than $1$, and let $p(i)$ denote the parent of $i \in S \setminus \{s\}$ in $T=T(\sigma, s)$. If $i<\sigma(i)$, then 
$$
p(i)= \max\{\sigma(i), \sigma^2(i), \ldots, \sigma^{k}(i)\}, \mbox{ where } k+1=\min\{j>0 : i\geq \sigma^j(i)\}.
$$
If $i > \sigma(i)$, then 
$$
p(i)= \min\{\sigma(i), \sigma^2(i), \ldots, \sigma^{k}(i)\}, \mbox{ where } k+1=\min\{j>0 : \sigma^j(i)\geq i\}.
$$
\end{proposition} 
\begin{proof}
Let $p(i)$ denote the parent of $i \in S \setminus \{s\}$ in $T=T(\sigma, s)$. Note that either 
\begin{itemize}
\item[(a)]
$i$ is one of the $x_j$'s in \eqref{ws}, or;
\item[(b)]
the parent of $i$ in $T(\sigma, s)$ is also the parent of $i$ in $T(W_jx_j, x_j)$ for the word $W_j$ which $i$ is a letter of. 
\end{itemize}
In case (a) the description of $p(i)$ is obviously correct. But then it is also true in general (by induction on $|S|$) in light of (b).  
\end{proof}
It will be convenient to depict (the tree associated to) a marked cycle as a \emph{diagram} of arcs: Order $S$  on a line. If $i<j$ and $i$ is a child of $j$, we draw an arc between $i$ and $j$ \emph{above} the line.  If $i<j$ and $j$ is a child of $i$, we draw an arc between $i$ and $j$ \emph{below} the line. Hence to each marked cycle we associate a diagram, see Fig.~\ref{Figdi}.

\begin{figure}[htp]
\begin{center}
\includegraphics[width=7cm]{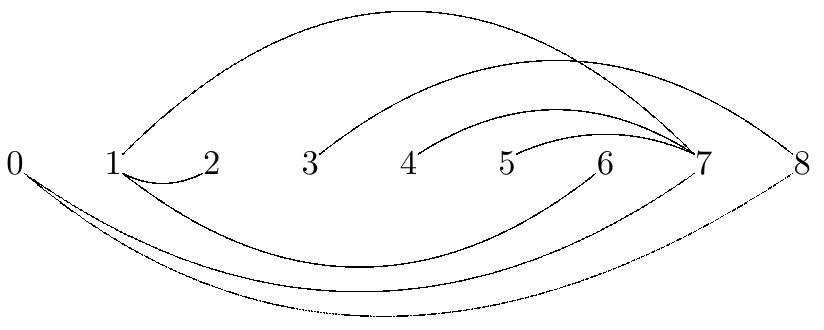} 
\end{center}
\caption{ The diagram of the marked cycle $((386214570), 0)$, or the permutation $\sigma=
(38)(621457)$. } 
\label{Figdi}
\end{figure}

Let $\OO_n$ be the set of alternative trees on $[0,n+1]:=\{0,1,\ldots, n+1\}$ with $0$ as a root. The set $\OO_n$ is in bijection with $\sym_{n+1}$ by the map $T' : \sym_{n+1} \rightarrow \OO_n$ defined as follows. Let $\sigma_1, \ldots, \sigma_k$ be the cycles of $\pi \in  \sym_{n+1}$, and let $s_i$ be the maximal element of $\sigma_i$ for each $1 \leq i \leq k$. Then the root of $T'(\pi)$ is $0$, and the children of the root are $s_1, \ldots, s_k$. The maximal subtree with root $s_i$ is defined to be 
$T(\sigma_i,s_i)$, where $T$ is defined above, see Fig.~\ref{AltTree}.

Hence 
$
p(i) = 0$ if $i$ is the maximal element in its cycle, and if $i<\sigma(i)$, then by Proposition~\ref{cycleprop}
$$
p(i)= \max\{\sigma(i), \sigma^2(i), \ldots, \sigma^{k}(i)\}, \mbox{ where } k+1=\min\{j>0 : i\geq\sigma^j(i)\}, 
$$
and if $i>\sigma(i)$, then 
$$
p(i)= \min\{\sigma(i), \sigma^2(i), \ldots, \sigma^{k}(i)\}, \mbox{ where } k+1=\min\{j>0 : \sigma^j(i)\geq i\}. 
$$

Recall that $i \in [n]$ is an \emph{excedance} of $\sigma \in \sym_{n+1}$ if $\sigma(i)>i$. Let $\XX(\sigma)$ denote the set of excedances of $\sigma$. 
A non-root vertex $i$ of a alternative tree $\mathcal T$ is an \emph{excedance} if 
$i$ is smaller than its parent $p(i)$.  Define three statistics on alternating trees. A pair $1\leq i<j \leq n$ is \emph{yin-yang} in $\mathcal T$ if $p(j)<i<j<p(i)$, see Fig.~\ref{picyy}. Let $\yy(\mathcal T )$ denote the number of yin-yang pairs in $\mathcal T$.  Let $c_0(\mathcal T)$ be the number of children of $0$, and $c_1(\mathcal T)$  the number of children of $n+1$. 

\begin{figure}[htp]
\setlength{\unitlength}{1.0cm}
\begin{picture}(4,2)(0,-.5)
\put(0,0){\circle*{0.2}} 
\put(1,0){\circle*{0.2}}
\put(2,0){\circle*{0.2}}
\put(3,0){\circle*{0.2}}
\put(0.2,0.5){\makebox(0,0){$p(j)$}}
\put(1.3,0){\makebox(0,0){$i$}}
\put(2.3,0){\makebox(0,0){$j$}}
\put(3.2,0.5){\makebox(0,0){$p(i)$}}
\qbezier(0,0)(1,-1)(2,0) 
\qbezier(1,0)(2,1)(3,0) 
\end{picture}
\caption{ A yin-yang pair $(i,j)$ in the diagram of a permutation/tree. } 
\label{picyy}
\end{figure}
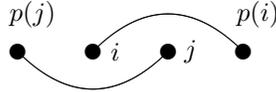

Define $\XX(\mathcal T) \in \{0,1\}^n$ by $\XX(\mathcal T)(i)=1$ if and only if $i$ is an excedance in $\mathcal T$, and define a map $[\cdot]: \{0,1\}^* \rightarrow \RR[a,b,q]$ by 
\begin{equation}\label{brack1}
[\eta] = \sum a^{c_0(\mathcal T)}b^{c_1(\mathcal T)}q^{\yy(\mathcal T)}, 
\end{equation}
where the sum is over all $\mathcal T \in \cup_{n \geq 0} \OO_n$ with $\XX(\mathcal T) = \eta$. 

Let $\sigma = Ws$, $s=n+1$, be the (marked) cycle of $\pi \in \sym_{n+1}$ containing $n+1$. Then
\begin{align*}
c(\pi) &:= \# \mbox{ cycles of } \pi, \\
c'(\pi) &:= \# \mbox{ right-to-left minima of } W \\
\yy(\pi) &:= \yy(T'(\pi)).
\end{align*}
Note that $\yy(\pi)$ may be  intrinsically defined by the description of $p\circ T'$   above. 

The following theorem is now an immediate consequence of the above bijection. 

\begin{theorem}
Let $\mathcal T=T'(\pi)$, where $\pi \in \sym_{n+1}$. Then 
$$(\XX(\mathcal T), c_0(\mathcal T), c_1(\mathcal T), \yy(\mathcal T)) = (\XX(\pi), c(\pi), c'(\pi), \yy(\pi)),$$
where we consider $\XX(\pi)$ as an element of $\{0,1\}^n$. 
\end{theorem}

\begin{theorem}\label{solsn}
Let $u,v \in \{0,1\}^*$ and let $[\cdot]$ be defined as in \eqref{brack1}. Then 
$$
[u10v]=q[u01v] +[u1v]+[u0v] \quad \mbox{ and } \quad [0v]=a[v], [u1] = b[u].
$$
\end{theorem}
\begin{proof}
Two vertices in a tree are \emph{comparable} if one of them is a descendant of the other.  Let $i=|u|+1$ and 
suppose $\eta(i) \neq \eta(i+1)$. 
We claim that if $i$ and $i+1$ are comparable, then 
\begin{itemize}
\item[(a)] $i$ is a leaf, $p(i)=i+1$ and $p(i+1)<i$, or;
\item[(b)] $i+1$ is a leaf, $p(i+1)=i$ and $p(i)>i+1$. 
\end{itemize}
To prove the claim first note that $i$ and $i+1$ cannot be comparable if $\eta(i)=0$ and $\eta(i+1)=1$, since then $i$ is larger than it's descendants while $i+1$ is smaller than it's descendants. Consider the case when $\eta(i)=1$ and $\eta(i+1)=0$. Then $p(i)>i$ and  $p(i+1)<i+1$.  If $i$ is a descendent of $i+1$, then $i$ has to be a leaf since $i+1$ is larger then it's descendants while $i$ is smaller than all it's descendants. But then $p(i)=i+1$, since otherwise $p(i)>i+1$ which means that $i+1$ has a descendant which is larger than $i+1$. By a similar argument if $i+1$ is a descendant of $i$, then $i+1$ is a leaf  and $p(i+1)=i$.

For $\eta \in \{0,1\}^*$, let $\OO(\eta)=\{\mathcal T \in \cup_{n \geq 0} \OO_n : \XX(\mathcal T)=\eta\}$. Hence in both (a) and (b) above we have $\mathcal T \in \OO(u10v)$. 
Define a map 
$\phi : \OO(u10v) \rightarrow \OO(u01v) \cup  \OO(u1v) \cup \OO(u0v)$ as follows. 

If $i$ and $i+1$ are non-comparable, then $\phi(\mathcal T)$ is obtained by switching the labels $i$ and $i+1$ in the tree. By the above claim this is a bijection between set of trees in $\OO(u10v)$ for which $i$ and $i+1$ are non-comparable and $\OO(u01v)$. Moreover, the yin-yang 
$p(i+1)<i<i+1<p(i)$ is destroyed by $\phi$ and no other yin-yangs are destroyed or created.

If $i$ and $i+1$ are comparable, then $\phi(\mathcal T)$ is obtained by contracting the edge between $i$ and $i+1$ while keeping the label $i$ and then relabel the vertices with $[0,n]$ so that the relative order is preserved. Then $\phi$ is a bijection between the set of trees satisfying 
(a) and $\OO(u0v)$, and $\phi$ is a bijection between the set of trees satisfying 
(b) and $\OO(u1v)$. No yin-yangs are created or destroyed. This establishes the first equation. 

If $1$ is not an excedance, then $1$ is a leaf and a child of the root. Hence we may contract the edge between  $0$ and $1$ and relabel the vertices, which shows $[0v]=a[v]$. 

Similarly, if $n$ is an excedance, then $n$ is a leaf and a child of $n+1$. Hence we may contract this edge and relabel the vertices, which proves $[u1]=b[u]$. 
\end{proof} 

\begin{theorem}\label{asep}
The multivariate partition function of the stationary distribution of the ASEP on $n$ sites with parameters $\alpha, \beta >0$, $q\geq 0$ and $\gamma=\delta=0$ is a constant multiple of 
$$
\sum_{\pi \in \sym_{n+1}} \alpha^{-c(\pi)} \beta^{-c'(\pi)} q^{\yy(\pi)} \prod_{i \in \XX(\pi)} x_i. 
$$
\end{theorem}

\begin{proof}
The theorem follows immediately by comparing Theorem~\ref{az} and Theorem~\ref{solsn}. 
\end{proof}

Note that a different combinatorial interpretation in terms of permutation statistics of the univariate partition function (equivalent to setting all $x_i$'s equal in Theorem~\ref{asep}) of the ASEP on $n$ sites with parameters $\alpha, \beta >0$, $q\geq 0$ and $\gamma=\delta=0$ was recently obtained by Josuat-Verg\`es \cite{J-V}. 

\section{Decorated alternative trees and colored permutations}
Let $r$ be a positive integer. Consider 
$$
\ZZ_r \wr \sym_n=\{ (\kappa, \sigma) \mid \kappa : [n] \rightarrow \ZZ_r \mbox{ and } \sigma \in \sym_n\}, 
$$ the \emph{wreath product} of the symmetric group of order $n$ with a cyclic group of order $r$. The elements
of the $\ZZ_r\wr \sym_n$ are often referred to as $r$-\emph{colored permutations}. 

There are several different ways of defining excedances for wreath products. For our purposes, the definition of Steingr\'imsson \cite{Stein94} is the most suitable choice.
Let $\pi =(\kappa, \sigma) \in \ZZ_r \wr \sym_n$.  Define the \emph{excedance set}, $\XX(\pi)$, and the  \emph{anti-excedance set}, $\YY(\pi)$, by 
$$i \in \XX(\pi) \ \ \ \ \mbox{ if and only if } \ \ \ \  \begin{cases} 
\sigma(i) > i, \mbox{ or}; \\
\sigma(i)=i \mbox{ and } \kappa_i \neq 0. 
\end{cases} 
$$
and 
$$\sigma(i) \in \YY(\pi)  \ \ \ \  \mbox{ if  and only if } \ \ \ \  \begin{cases} 
\sigma(i) < i, \mbox{ or}; \\
\sigma(i)=i \mbox{ and } \kappa_i=0. 
\end{cases} 
$$

Let $\pi = (\kappa , \sigma) \in \ZZ_r \wr \sym_n$ and consider the cycle decomposition of $\sigma$. A cycle $c$ of $\sigma$ is called a \emph{zero cycle} if $\kappa_i=0$ for the maximal element $i$ of $c$, otherwise $c$ is called a \emph{non-zero cycle}.

 A \emph{decorated alternative tree} on a finite non-empty set of integers $S$ is an alternative tree on $S$ where the vertices are also colored with $0, \ldots, r-1$, where $r>1$. Hence each vertex in the tree is labeled with an element from $S\times \{0,\ldots, r-1\}$. The coloring should  obey the following restrictions. Let $s$ and $t$ be the smallest and largest vertex (with respect to the total order on $S$), respectively. 
 \begin{itemize}
 \item[(a)] The children of the root (which is $s$) all have color zero. 
 \item[(b)] The largest vertex of a maximal subtree whose root is a child of $t$ has non-zero color.  
 \item[(c)] The root $s$ has color $1$, and $t$ has color $0$. 
 \end{itemize}

 Let $\OO_n^r$ be the set of decorated trees on $S=[0,n+1]$ with permitted colors $\{0, \ldots, r-1\}$. Then 
 $\OO_n^r$ is in bijection with $\ZZ_r \wr \sym_n$ by the mapping $T''$ described below. Start with the tree consisting of a single vertex $0$ (as a root) and attach $n+1$ to it. Give them colors $1$ and $0$, respectively. 
 Let $\sigma_1, \ldots, \sigma_k$ be the zero cycles, and let $s_i$ be the maximal element of $\sigma_i$. The children of $0$ are $n+1$ and $s_1, \ldots, s_k$. The maximal subtree with root $s_i$ is defined to be $T(\sigma_i,s_i)$, where $T$ is defined above. Now assign colors according to $\kappa$. 
 
  Let $\tau_1, \ldots, \tau_\ell$ be the non-zero cycles, and let $t_i$ be the minimal element of
$\tau_i$. The children of $n+1$ are $t_1, \ldots, t_\ell$. 
  The maximal subtree with root $t_i$ is defined to be $T(\tau_i,t_i)$, where $T$ is defined above.
Assign colors according to $\kappa$.  See Fig.~\ref{DecTree} for an example. 
\begin{figure}[htp]
$$
\begin{tikzpicture}[scale=1, inner sep=2pt]
  \node (0) at (2,3) {$0^1$};
  \node (1) at (4.5,1) {$7^4$};
    \node (2) at (2.5,1) {$2^2$};
    \node (3) at (3.5,2) {$1^8$};
    \node (4) at (4,0) {$4^2$};
    \node (5) at (5,0) {$5^0$};
  \node (6) at (3.5,1) {$6^1$};
  \node (7) at (1,2) {$3^0$};
   \node (8) at (3.5,2.8) {$8^0$};
\draw (0)--(7);  
\draw(0) -- (8) --(3) -- (2);
\draw (3) -- (6);
\draw (3) -- (1);
\draw (1) -- (5);
\draw (1) -- (4);
\end{tikzpicture}
$$ 
\caption{ An alternative tree on the set $[0,8],$ with root $0,$ corresponding to the colored permutation $\pi=(3^0)(4^2 5^0 7^4 6^1 2^2 1^8)$. } 
\label{DecTree}
\end{figure}
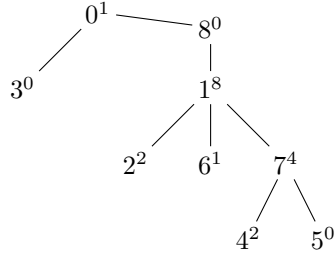
  
The definitions of $\XX$ and $\yy$ are the same as for the non-decorated trees. For a decorated tree $\mathcal{T} \in \OO_n^r$,  $c_0(\mathcal{T})$ is defined as the number of children of 0 minus one, while $c_1(\mathcal{T})$ is the number of children of $n+1$. Define $[\cdot]_r: \{0,1\}^* \rightarrow \RR[a,b,q]$ by 
\begin{equation}\label{brackr}
[\eta]_r = \sum a^{c_0(\mathcal T)}b^{c_1(\mathcal T)}q^{\yy(\mathcal T)}, 
\end{equation}
where the sum is over all $\mathcal T \in \cup_{n \geq 0} \OO_n^r$ with $\XX(\mathcal T) = \eta$. 

\begin{theorem}\label{dequi}
Let $\mathcal T=T''(\pi)$, where $\pi \in \ZZ_r \wr \sym_n$. Then 
$$(\XX(\mathcal T), c_0(\mathcal T), c_1(\mathcal T), \yy(\mathcal T)) = (\XX(\pi), \nc(\pi), \zc(\pi), \yy(\pi)),$$
where we consider $\XX(\pi)$ as an element of $\{0,1\}^n$. 
\end{theorem}

\begin{theorem}\label{slow} Let $u,v \in \{0,1\}^*$, and let $[\cdot]_r$ be defined as in \eqref{brackr}. Then
$$
[u10v]_r=q[u01v]_r+ r[u1v]_r+r[u0v]_r, \quad [0v]_r=a[v]_r, \quad [u1]_r=b(r-1)[u]_r.
$$
\end{theorem}

\begin{proof}
The definition of $\phi : \OO(u10v) \rightarrow \OO(u01v) \cup  \OO(u1v) \cup \OO(u0v)$ is almost the same as for the non-decorated case, we just have to specify how the colors are effected. For the non-comparable case, the color of $i$ and $i+1$ are swapped, so that the color stay at the same place in the tree. When we contract an edge in the comparable case we keep the color of $i+1$. Hence the map in the comparable case is $r$ to $1$, which explains the factor $r$ in the equation. 

If $1$ is not an excedance, then $1$ is a leaf, has color zero,  and $1$ is  a child of the root. Hence we may contract the edge between $0$ and $1$ and relabel the vertices, which shows $[0v]_r=a[v]_r$. 

Similarly, if $n$ is an excedance, then $n$ is a leaf of non-zero color, and $n$ is a child of $n+1$. Hence we may contract this edge and relabel the vertices, which proves $[u1]_r=b(r-1)[u]_r$. 
\end{proof}

\begin{theorem}\label{maincolor}
The multivariate partition function of the stationary distribution of the ASEP on $n$ sites with parameters $\alpha, \beta >0$, $q\geq 0$ and $\gamma=\delta=0$ is a constant multiple of 
$$
\sum_{\sigma \in \mathbb{Z}_r \wr \mathfrak{S}_{n}} 
\left(\frac{r}{\alpha}\right)^{\nc(\sigma)} 
\left(\frac{r}{(r-1)\beta}\right)^{\zc(\sigma)} q^{\yy(\sigma)} \prod_{i \in \XX(\sigma)} x_i, 
$$
where $r\geq 2$ is an integer. 

In particular, when $\alpha = r$ and $\beta = r/(r-1)$, then the partition function is a constant multiple of 
$$
\sum_{\sigma \in \mathbb{Z}_r \wr \mathfrak{S}_{n}} 
q^{\yy(\sigma)} \prod_{i \in \XX(\sigma)} x_i, 
$$
\end{theorem}

\begin{proof}
The theorem follows immediately by using Theorem \ref{dequi}, and comparing Theorem~\ref{az} with Theorem~\ref{slow}. 
\end{proof}

\section{Multivariate Eulerian polynomials and stability}

The Eulerian polynomials are important in enumerative and algebraic combinatorics and their generalizations to finite 
 Coxeter groups, wreath products and partially ordered sets have been studied frequently. An important property, first noted by Frobenius \cite{Fro},  is that all zeros of $A_n(x)$ are real. This result has subsequently been lifted to different generalizations of Eulerian polynomials \cite{Brenti, SV13}. 
 
 Recently efforts have been made to generalize Frobenius' result in yet another direction, namely to multivariate polynomials \cite{VW13}. A notion of  ``real-rootedness'' that has been fruitful in several settings is the following. A polynomial $P(x_1, \ldots, x_n)$ is \emph{stable} if $P(x_1, \ldots, x_n) \neq 0$ whenever $\Im(x_j)>0$ for all $1 \leq j \leq n$. For applications to the ASEP we find it convenient to define the multivariate generalization in terms of excedances. 

There is a strong relationship between symmetric exclusion processes and stability  which was first proved in \cite{BBL} (without (B) and (D)) and in \cite{Wa3} (with (B) and (D)). 

\begin{theorem}\label{prest}
Consider the Markov chain described by (J), (B) and (D) in Section~\ref{secEP}, with $Q$ symmetric. If the partition function of the initial distribution is stable, then the distribution is stable for all $t \geq 0$. 
\end{theorem}

\begin{corollary}\label{stabcor}
Consider the Markov chain described by (J), (B) and (D) in Section~\ref{secEP}, with $Q$ symmetric. If the Markov chain is irreducible and positive recurrent, then the  partition function of the (unique) stationary distribution is stable. 
\end{corollary}

\begin{proof}
Choose an initial distribution with stable partition function, for example a product measure. Then the partition function, $Z_t(\xx)$, of the distribution at time $t$ is stable for all $t>0$ by Theorem \ref{prest}. The partition function of the stationary distribution is given by $\lim_{t \to \infty}Z_t(\xx)$, and hence the corollary follows from Hurwitz' theorem on the continuity of zeros (see \cite[Footnote 3, p.~96]{COSW} for a multivariate version). 
\end{proof}

As an immediate corollary of Theorem \ref{maincolor} and Corolllary \ref{stabcor} we have. 

\begin{corollary}\label{consta}
Let $n$ and $r$ be positive integers and $a$ and $b$ nonnegative real numbers. Then the polynomial 
\begin{equation}\label{brackxi}
 \sum_{ \pi \in \ZZ_r \wr \sym_n}a^{\nc(\pi)}b^{\zc(\pi)}\prod_{i \in \XX(\pi)}x_{i}, 
\end{equation}
where $\zc(\pi)$ and $\nc(\pi)$ denotes the number of zero cycles and non-zero cycles of $\pi$, respectively,  is stable. 
\end{corollary}

This corollary generalizes a recent theorem of Hitczenko and Janson \cite[Theorem 4.5]{HiJa} who proved that the univariate polynomials obtained by setting all the $x_i$'s in \eqref{brackxi} equal are real--rooted. 
We shall now see how Corollary \ref{consta}  can be generalized further by introducing a new set of variables. 

Let 
$$
F_n = F_{n,r}(\xx,\yyy, a,b)=  \sum_{ \pi \in \ZZ_r \wr \sym_n}a^{\nc(\pi)}b^{\zc(\pi)}\prod_{i \in \XX(\pi)}x_{i} \prod_{j \in \YY(\pi) } y_j. 
$$

\begin{theorem}\label{difff}For positive integers $n$ and $r$,
\begin{equation}\label{nrxyab}
F_n = (a(r-1)x_1 + by_1)F^*_{n-1}+ rx_1y_1\sum_{j=2}^{n} \left(\frac{ \partial }{\partial x_j} + \frac{ \partial }{\partial y_j}\right) F^*_{n-1}, 
\end{equation}
where $F^*_{n-1}$ is obtained from $F_{n-1}$ by the changes of variables $x_i \to x_{i+1}$ and $y_i \to y_{i+1}$ for all $1\leq i \leq n-1$, and $F_0 =1$. 
\end{theorem}

\begin{proof} Consider a colored permutation $\pi^* = (\kappa^*,\sigma^*)$, where $\sigma^*$ is a permutation of $\{2,\ldots, n\}$ and $\kappa^* : \{2, \ldots, n\} \to \ZZ_r$. When $\sigma^*$ is written as a product of cycles and $i \mapsto j$ (i.e., $\sigma^*(i) = j$), then we record $x_i$ if $i<j$ and $y_j$ if $j<i$. If $i \mapsto i$ is a fixed point, then  we record $x_i$ if $\kappa^*_i \neq 0$, and $y_i$ otherwise. 

To create an element $\pi=(\kappa, \sigma) \in \ZZ_r \wr \sym_n$ from $\pi^*$ we insert $1$ into $\sigma^*$ and choose its color $\kappa_1$.
Hence inserting $1$ between $i \mapsto j$ in an existing cycle will have the effect of taking the derivative with respect to  $x_i$ or $y_j$ depending on
whether $i \in \XX(\pi^*)$ or $ j \in \YY(\pi^*)$, respectively,
and multiplying by $x_1y_1$ since $1 \in \XX(\pi) \cap \YY(\pi)$.
Note also that the assignment of variables to the arrows $i \mapsto j$ is injective. This explains the second term on the right hand side of \eqref{nrxyab}. If we make $1$ a fixed point we either create a new non-zero cycle and an excedance (if $\kappa_1 \neq 0$), or a zero cycle and an anti-excedance (if $\kappa_1 =0$). This explains the first term on the right hand side of \eqref{nrxyab}.
\end{proof}

\begin{remark} Theorem \ref{difff}  can be seen as a relation satisfied by the stationary distributions of the ASEP with $q-1=\gamma=\delta =0$. Is there a similar relation for general $\alpha, \beta, \gamma, \delta$? 
\end{remark}

\begin{remark}
In \cite{VW13} a multivariate extension of the Eulerian polynomials for wreath products was defined in terms of descent-- and ascent bottoms. 
By the recursion given in the proof of Theorem 3.15 in \cite{VW13}, we see that for $a=b=1$, their polynomials are the same as ours up to a reindexing of the variables. 
\end{remark}

\begin{theorem}\label{stabF}
Let $r>1$ and $n$ be positive integers and $a,b \geq 0$. Then $F_{n,r}(\xx,\yyy, a,b)$ is stable. 
\end{theorem}

\begin{proof}
By \eqref{nrxyab}, 
$
F_n = T(F_{n-1}^*)
$, 
where $T$ is the linear operator
$$
T= a(r-1)x_1 + by_1+ rx_1y_1\sum_{j=2}^{n} \left(\frac{ \partial }{\partial x_j} + \frac{ \partial }{\partial y_j}\right).
$$
 Since $F_1= a(r-1)x_1+by_1$ is obviously stable it remains to prove that $T$ preserves stability. By the characterization of stability preservers, \cite[Theorem~2.2]{LYPSI}, this is the case if the polynomial 
$$
G_T = T\left( (x_1+z_1)\cdots (x_n+z_n)(y_1+w_1)\cdots (y_n+w_n)\right),
$$
is stable (in $4n$ variables). Here $T$ acts on the $x$- and $y$-variables 
and treats the $z$- and $w$-variables as constants. Now 
$$
\frac { G_T } { x_1y_1 \prod_{j=1}^n(x_j+z_j)(y_j+w_j)}= \frac {a(r-1)}{y_1}+ \frac b {x_1}+ \sum_{j=2}^n \left( \frac 1 {x_j + z_j} + \frac 1 {y_j + w_j} \right). 
$$
Each term on the right hand side of the above equation has negative imaginary part whenever all variables have positive imaginary parts. Hence $G_T$ is stable and the theorem follows. 
\end{proof}

\section{Negative dependence}
Negative dependence in probability theory models repelling particles. There are many correlation inequalities of varying strength that model negative dependence, see \cite{BBL,Pem}. For example, a discrete probability measure $\mu$  on $\{0,1\}^n$ is \emph{negatively associated} if 
\begin{equation}\label{na}
\int f g d\mu \leq \int f d\mu \int g d\mu,
\end{equation}
whenever $f,g : \{0,1\}^n \rightarrow \RR$ depend on disjoint sets of variables (i.e., $f$ depends only on  $\{\eta_i : i \in A\}$ and $g$ depends only on  $\{\eta_j : j \in B\}$, where $A\cap B = \emptyset$).  In particular if $\mu$ is negatively associated, then it is \emph{pairwise negatively correlated}, i.e., for distinct $i,j \in [n]$: 
$$
\mu(\eta(i)=\eta(j)=1) \leq \mu(\eta(i)=1) \mu(\eta(j)=1),
$$
which is obtained from \eqref{na} by setting 
$$
f(\eta) = \begin{cases} 1 &\mbox{ if } \eta(i)=1 \\
                                 0 &\mbox{ otherwise}
                                 \end{cases} \ \ \  \mbox{ and } \ \ \ 
                                 g(\eta) = \begin{cases} 1 &\mbox{ if } \eta(j)=1 \\
                                 0 &\mbox{ otherwise}
                                 \end{cases}.
                                 $$                                 
It was proved in \cite{BBL} that if the multivariate partition function of a discrete probability measure $\mu$ is stable (such measures are called \emph{strong Rayleigh}), then it satisfies several of the strongest correlation inequalities known to model negative dependence. In particular $\mu$ is negatively associated. 
\begin{corollary}
The stationary distribution of the ASEP with $q=1$ and $\alpha,\beta, \gamma, \delta \geq 0$ is negatively associated. 
\end{corollary}
Hitczenko and Janson \cite{HiJa} used the real--rootedness of the univariate partition function (obtained by setting $x_1=\cdots=x_n=x$ in \ref{partitioneq}) of the stationary distribution of the ASEP with $q=1$, $\alpha,\beta \geq 0$ and $\gamma=\delta= 0$, to prove concentration inequalities for the corresponding measures. Since we now know that the multivariate partition functions are stable whenever $q=1$ and $\alpha,\beta, \gamma, \delta \geq 0$ there are  more refined concentration inequalities available due to Pemantle and Peres \cite{PePe}. A function 
$f : \{0,1\}^n \rightarrow \RR$ is \emph{Lipschitz-1} if 
$$
|f(\eta)-f(\xi)| \leq d(\eta,\xi), \ \ \mbox{ for all } \eta, \xi \in \{0,1\}^n, 
$$
where $d$ is the \emph{Hamming distance}.  
\begin{theorem}[Pemantle and Peres, \cite{PePe}]
Suppose $\mu$ is a probability measure on $\{0,1\}^n$ whose partition function is stable and has mean $m = \mathbb{E}(\sum_{i=1}^n\eta_i)$. If $f$ is any Lipschitz-1 function on $\{0,1\}^n$, then  
$$
\mu( \eta : |f(\eta)-\mathbb{E}f|>a) \leq 5 \exp\left( - \frac {a^2}{16(a+2m)}\right).
$$
\end{theorem}
The case when $f(\eta)=\sum_{i=1}^n \eta_i$ (the number of particles) corresponds to the univariate partition function. 

We pose as an open problem to investigate negative dependence properties when $q \neq 1$. In particular:
\begin{question}\label{stbq}
Consider the ASEP on $n$ sites with $\gamma=\delta=0$, $\alpha=\beta=1$ and $q \neq 1$. Is the multivariate partition function of the stationary distribution stable? Is the stationary distribution negatively associated?
\end{question}
Question \ref{stbq} is open even for the case $q = 0$. It is also open whether the univariate partition function is real--rooted, for $\gamma=\delta=0$, $\alpha=\beta=1$ and $q \neq 1$.
However, for $q=0$, we get the $(n+1)$st \emph{Narayana polynomial}, which is known to be real--rooted. This can be seen as supporting evidence for a affirmative answer to Question \ref{stbq} when $q=0$. 

It would be interesting if one could find explicit combinatorial models for the exclusion process for classes of labeled graphs which are not necessarily lines. In the symmetric case the partition functions of the stationary distributions (if unique) will be stable by Corollary \ref{stabcor}, and thus the stationary distributions will be negatively associated.

\end{document}